\documentclass[leqno,a4paper]{article}
\usepackage{amsmath,bbm,amsthm}

\newcommand{\cD}{\mathcal{D}}

\newcommand{\R}{{\mathbbm R}}

\newcommand{\vp}{\varphi}

\newcommand{\m}{{\mu}}

\newtheorem{definition}{Definition}[section]
\newtheorem{remark}{Remark}[section]
\newtheorem{theorem}{Theorem}[section]
\newtheorem{lemma}[theorem]{Lemma}
\newtheorem{proposition}{Proposition}[section]
\newtheorem{corollary}{Corollary}[section]
\numberwithin{equation}{section}
\def\vs1{\vspace{1ex}}

\def\O{\Omega}
\def\pa{\partial}
\def\dy{\displaystyle}

\def\ve{\varepsilon}
\def\be{\begin{equation}}
\def\ba{\begin{array}}
\def\ea{\end{array}}
\def\ee{\end{equation}}

\begin{document}
\title{\bf\large On the global regularity for nonlinear systems \\
of the p-Laplacian type}
\author{ H.~Beir\~ao da Veiga\ and \ F.~Crispo}

\date{}
\maketitle
\begin{abstract}
We consider the Dirichlet boundary value problem for nonlinear
systems of partial differential equations with $p$-structure. We
choose two representative cases: the ``full gradient case'',
corresponding to a p-Laplacian, and the ``symmetric gradient case'',
arising from mathematical physics. The domain is either the so
called``cubic domain'' or a bounded open subset of $\R^3$ with a
smooth boundary. We are interested in regularity results, up to the
boundary, for the second order derivatives of the velocity field.
Depending on the model and on the range of $p$, $p<2$ or $p>2$, we
prove different regularity results. It is worth noting that in the
full gradient case, with $p<2\,,$ we cover the degenerate case, and
obtain $W^{2,q}$-global regularity results, for arbitrarily large
values of $q$.
\end{abstract}

\vspace{0.2cm}




 \section{Introduction}
We are concerned with the regularity problem for solutions of
nonlinear systems of partial differential equations with
$p$-structure, $p\,>1$, under Dirichlet boundary conditions. In
order to emphasize the main ideas we confine ourselves to the
following representative cases (where  $\mu\,\geq \,0$ is
a fixed constant):\par%
The ``full gradient case'' \be\label{NSC}%
-\,\nabla
\cdot S\,(\,\nabla\, u\,) =\,f\,,
\ee where \be\label{ex1}
S(\nabla\,u)=\,(\,\m+|\,\nabla\,u|\,)^{p-2}\,\nabla\,u\,;\ee%
and the  ``symmetric gradient case'' \be\label{NSCS} -\,\nabla \cdot
S\,(\,\cD\, u\,) =\,f\,,
 \ee where \be\label{ex2}
S(\cD\,u)=\,(\,\m+|\,\cD\,u\,|)^{p-2}\,\cD\,u\,.\ee%
As usual, $$\cD\, u\,=\,\frac 12\,(\,\nabla\, u\,+\,\nabla
\,u^T\,)$$
is the symmetric part of the gradient of $u$.\par%
When $\mu=0$ in \eqref{ex1}, the system \eqref{NSC} is the
well-known p-Laplacian system.\par%
It is worth noting that our results concern global (up to the
boundary), full regularity for the second derivatives of solutions
to the previous systems, with Dirichlet boundary conditions (one
could also consider slip type boundary conditions). The regularity
issue for systems like \eqref{NSC}  has received substantial
attention, mostly concerned with an equation in place of a system,
and with $C^{1,\alpha}_{loc}$-regularity. In the scalar case,
existence and interior integrability of the second derivatives are
shown in \cite{Tolk2}, for any $p>1$;  in \cite{Liu} the regularity
up to the boundary is obtained for any $p\,\in (1,2)$. For systems
(solutions are $N$-dimensional vector fields, $N>1$), we recall
\cite{acerbi} for $p\in (1,2)$, \cite{GM} and \cite{uhlenb} for
$p>2$, and \cite{Hamburger} for any $p>1$. These papers deal only
with homogeneous systems and the techniques, sometimes quite
involved, seem not to be directly applicable to the non-homogeneous
setting. In particular, \cite{acerbi} is the only paper in which the
$L^2$-regularity of second derivatives is considered. However, the
results are shown only in the interior. Therefore our results seem
to be the first regularity results, up to the boundary, for the
second derivatives of solutions. Another main difference with the
above papers is that we do not require differentiability of $S$, but
merely Lipschitz continuity.
 For related results and for an extensive bibliography we also
refer to papers \cite{AM}, \cite{DB}, \cite{DBM}, \cite{FuM},
\cite{FS},  \cite{Lieb}, \cite{MP}, \cite{Ming} and references
therein.
\par%
We have not found papers dealing with the equations arising from the
choice \eqref{ex2} for $S$. This kind of model is used in various
branches of mathematical physics as, for instance, in non-linear
elasticity or in non-linear diffusion. Actually, our interest in
systems \eqref{NSC} and \eqref{NSCS} arises from our previous
studies on fluid dynamics problems. Indeed, we recall that a good
model for non-Newtonian fluids with shear dependent viscosity is the
following one \be\label{NNF} -\,\nabla \cdot
\left[\,(\,\m+|\,\cD\,u\,|)^{p-2}\,\cD\,u\,\right]+(u\cdot\nabla)u+\nabla
\pi =\,f\,, \ \ \nabla\cdot u=0\,,\ee which can be obtained from
\eqref{NSCS} by adding the contribution of the pressure field $\pi$,
the convective term $(u\cdot\nabla)u$ and the divergence free
constraint. For this system, regularity up to the boundary  has been
considered in both the cases $p<2$ and $p>2$. The case $p=2$
corresponds to the well known Navier-Stokes system for Newtonian
fluids. For the more general regularity results and a wide
bibliography on this topic, we refer the reader to \cite{Bsmooth},
\cite{BKR} for $p>2$, and to \cite{Bsmooth2} for $p<2$. Despite many
contributions to the regularity issue, $W^{2,2}$-regularity up to
the boundary for solutions to \eqref{NNF} is still open, even for
the simplified setting of ``generalized'' Stokes system, obtained by
dropping the convective term in \eqref{NNF}. We mention the papers
\cite{crigri} and \cite{crigri2}, which, as far as we know, are the
only papers where the $W^{2,2}(\O)\cap
C^{1,\alpha}(\overline\Omega)$-regularity is obtained, under the
additional assumption of a small force. The regularity proved below
suggests that the main obstacle to the $W^{2,2}$-regularity of
solutions of \eqref{NNF} is actually the
presence of the pressure term.\par%
Our interest in fluid-mechanics, and in particular in non-Newtonian
fluids, leads us to consider the case $n=N=3$. However, it is worth
noting that our results can be immediately extended to dimensions
$n>3$, and to $N$-dimensional vector fields, $N\not=3$. Further, the
explicit choices \eqref{ex1} and \eqref{ex2} are done in order to
emphasize the core aspects of the results and to avoid additional
technicalities. Therefore, we do not consider a more general
dependence of $S$ on $\nabla u$ or $\cD\,u$, as for instance
$S(\nabla u)= \varphi(|\nabla u|)\, \nabla u$, under
 suitable assumptions on the scalar function $\varphi$. For the same
 reason we avoid the introduction of lower order terms.
\par In the sequel we cover both the cases $p<2$ and $p>2$, with,
however, some differences, and some restrictions on the exponent
$p\,$, as follows:
\par Case $p<2$: For $p<2$ we consider the ``full gradient
case'' \eqref{NSC}. In this case, all results hold also in the
degenerate case $\mu=0$. For any bounded and sufficiently smooth
domain $\O$, we prove $W^{2,q}(\O)$ regularity, for any $q\geq 2$.
Therefore, we get, as a by product, the H\"{o}lder continuity, up to
the boundary, of the gradient of the solution. Results are obtained
for $\,p\,$ belonging to suitable intervals $[C,2)$, where the
constants $C$ are defined precisely. \par%
Case $\,p>2$: We prove the $W^{2,2}$-regularity in both cases,
\eqref{NSC} and \eqref{NSCS}, provided that $\mu>0$. We restrict our
proofs to the ``cubic domain case'' (see the next section), where
the interesting boundary condition (Dirichlet) is imposed on two
opposite sides, and periodicity in the other two directions. This
choice, introduced in reference \cite{bvcubo} and used in a series
of other papers (see for instance \cite{BDV2,D,Crispo3,CrispoCh}),
is convenient in order to work with a flat boundary and, at the same
time, with a bounded domain. The main reason is that, in proving the
regularity theorem for $p>2$ (see Theorem \ref{main}), we apply the
difference quotients method: we appeal to translations parallel to
the flat boundary, and then retrieve the normal derivatives from the
equations. Then, the simplified framework of a cubic domain avoids
the need of localization techniques and changes of variables. The
results can be extended to smooth domains, by following
\cite{Bsmooth}, \cite{Bsmooth2}, and \cite{BKR}, where the extension
is done for the more involved system of non-Newtonian fluids (see
also \cite{MNR}). See also the Remark \ref{rempmag2}.

\section{Notation and statement of the main results}
Throughout this paper we denote by $\Omega$ a bounded
three-dimensional domain with smooth boundary, which we assume of
class $C^2$, and we consider the usual homogeneous Dirichlet
boundary conditions \be\label{diri} u_{|\partial\Omega}=0.\ee
Further, we denote by $Q$ the cube $Q=\,(\,]0,1[\,)^3$, and by
$\Gamma$ the two opposite faces of $Q$ in the $x_3$-direction, i.e.
$$\Gamma=\{\,x:\,|x_1|<1,\, |x_2|<1,\, x_3=\,0\,\}\,\cup\,\{\,x:\,|x_1|<1,\,
|x_2|<1,\, x_3=\,1\,\}.$$ We impose the Dirichlet boundary
conditions on $\Gamma$ \be\label{bc} u_{|\,\Gamma}\,=\,0, \ee
 and periodicity, with period equal to 1, in
both the $x_1$, $x_2$ directions.

\par By $L^p(\O)$ and $W^{m,p}(\O)$, $m$ nonnegative integer
and $p\in(1,+\infty)$, we denote the usual Lebesgue and Sobolev
spaces, with the standard norms $\|\cdot\|_{L^p(\O)}$ and
$\|\,\cdot\,\|_{W^{m,p}(\O)}$, respectively. We usually denote the
above norms by $\|\cdot\|_{p}$ and $\|\,\cdot\,\|_{m,p}$, when the
domain is clear. Further, we set $\|\cdot\|=\|\cdot\|_{2}$. We
denote by $W^{1,p}_0(\O)$ the closure in $W^{1,p}(\O)$ of
$C^\infty_0(\O)$ and by $W^{-1,p'}(\O)$, $p'=\,p/(p-1)$, the strong
dual of $W^{1,p}_0(\O)$ with norm $\|\,\cdot\,\|_{-1,p'}$. In
notation concerning duality pairings, norms and functional spaces,
we do not distinguish between scalar and vector fields.
\par We set $$V_p(\Omega)=\,\left\{\,v\in W^{1,p}\,(\O):\,
v_{|\pa\O}=0\,\right\},$$and
$$V_p(Q)=\,\left\{\,v\in W^{1,p}\,(Q):\, v_{|\Gamma}=0,
\, v \mbox{ is }x'\,-\,\mbox{periodic}\,\right\}.$$
 By $V_p'(\Omega)$ and $V_p'(Q)$ we denote the dual
spaces of $V_p(\Omega)$ and $V_p(Q)$, respectively.
\par
We use the summation convention on repeated indexes, except for the
index $s$. For any given pair of second order tensors $B$ and $C$,
we write $B\cdot C\equiv B_{ij}\,C_{ij}$.
\par
We denote by the symbols $c$, $c_1$, $c_2$, etc., positive constants
that may depend on $\mu$; by capital letters, $C$, $C_1$, $C_2$,
etc., we denote positive constants independent of $\mu
\geq\,0\,$(eventually, $\,\mu\,$ bounded from above). The same
symbol $c$ or $C$ may denote different constants, even in the same
equation.

\par%
 We
set $\partial_i \,u=\,\frac{\pa\, u}{\pa\, x_i}$,
$\partial_{ij}^2\,u=\,\frac{\pa^2\, u}{\pa\, x_i\pa\,x_j}$. Moreover
we set $(\nabla u)_{ij}=\pa_j\,u_i$ and $(\cD\,u)_{ij}=\frac 12
\left((\nabla u)_{ij}+(\nabla u)_{ji}\right)$. We denote by $D^2u$
the set of all the second partial derivatives of $u$. The symbol
$D^2_*u$ may denote any second-order partial derivative
$\pa_{hk}^2\,u\,$ except for the derivatives $\pa_{33}^2\,u\,$.
Moreover we set \be\label{SDstar} |\,D^2u\,|^2:=\sum_{i,j,k=1}
^3\!\!\left|\,\pa_{jk}^2\,u_i\,\right|^2\quad \mbox{ and } \quad
|\,D^2_*u\,|^2:=\sum_{i,j,k=1\atop (j,k)\not =(3,3)}
^3\!\!\left|\,\pa_{jk}^2\,u_i\,\right|^2\,.\ee%
We define the tensor $S(A)$ as \be\label{tensord}
S(A)=\,(\,\m+|\,A\,|)^{p-2}\,A\,,\ee with $\mu\,\geq \,0$ fixed
constant, $p\,>1$, and $A$ an arbitrary tensor field. It is easily
seen that $S(A)$ satisfies the following property: there exists a
positive constant $C_1$ such that
\begin{equation}\label{tensas}
\frac{\pa S_{i\,j}(A)}{\pa A_{k\,l}}\,B_{i\,j}\,B_{k\,l}\geq\,
C_1\,(\m+\,|\,A\,|)^{p-2}\,|\,B\,|^2\,,
\end{equation}
for any tensor $B$. Further \be\label{tensorS}(S(A)-S(B))\cdot
(A-B)\geq\, C_2\,\frac{|A-B|^2}{(\mu+|A|+|B|)}{\!\atop ^{{2-p}}}\,,\
\ee and \be\label{tensorS1}|\,S(A)-S(B)\,|\leq\,
C_3\,\frac{|A-B|}{(\mu+|A|+|B|)}{\!\atop ^{{2-p}}}\,,\ee for any
pair of tensors $A$ and $B$, with $C_2$ and $C_3$ positive
constants. The proof of the above estimates is essentially contained
in \cite{Giusti}. We also refer to \cite{DER} for a detailed proof.
\par Our aim is to prove the regularity results up to the boundary
given in the theorems below. Let us state our main results. We start
from the case $p>2$.
\begin{theorem}\label{main}
Assume that $p>2$ and $\mu>0$.  Let $f\in L^2(Q)$, and let $u\in
V_p(Q)$ be a weak solution of problem \eqref{NSC}--\eqref{bc} or of
problem \eqref{NSCS}--\eqref{bc}. Then $u\,\in W^{2,2}(Q)\,$.
Moreover, there is a constant $c$ such that%
\be\label{estsdp2}%
\|D^2\,u\,\|\leq c\,\|\,f\,\|\,.\ee
\end{theorem}
This theorem will be proved in the next section. \vskip0.1cm  The
other results concern the case $p<2$. Note that, in this case, the
parameter $\mu$ can be equal to zero, thus covering the
$p$-Laplacian systems. Further, here we consider a general smooth
bounded domain. On the other hand, we restrict our considerations to
the full gradient case.\par  Before stating the regularity theorems
for $p<2$, let us recall two well known inequalities for the Laplace operator.
The first, namely%
\be\label{lad2} \|\,D^2\,v\,\|\leq C_4\,\|\,\Delta\, v\,\|\,,\ee
holds for any function $v\in W^{2,2}(\O)\cap W_0^{1,2}(\O)\,$. Here
$C_4=C_4(\Omega)\,.\,$ Note that if $\O$ is a convex domain, then
$C_4=1$. For details we refer to \cite{Lad} (Chapter I, estimate
20). The second kind of estimates which we are going to use for a
$v\in W^{2,q}(\O)\cap W_0^{1,q}(\O)$, $q\geq 2$, is
\be\label{ladaq}\|D^2\,v\|_{q}\leq \,C_5\|\Delta v\|_q\,,\ee where
the constant $C_5$ depends only on $q$ and $\O$. It relies on
standard estimates for solution of the Dirichlet problem for the
Poisson equation.  Actually, there are two constants $K_1$ and
$K_2$, independent of $q$, such that%
\be\label{yud} K_1\, q\leq C_5\leq K_2\, q\,.\ee%
Similarly, one has \be\label{lad1q}\|\,\,v\,\|_{\,2,q}\leq
\,C\|\Delta v\|_q\,,\ee  where the constant $C$ depends on $q$ and
$\O$. For further details we refer to \cite{kos} and \cite{yud}.
\par
 For $p<\,2\,$ our main
results are the following.
\begin{theorem}\label{teorema}
Let be $\mu\geq 0$, and $\,1\,<p\,\leq\, 2\,$ such that
$(2-p)\,C_4<\,1\,$, where $C_4$ is given by \eqref{lad2}. Let $f \in
L^{\frac{6}{p+1}}(\Omega)$. Then, the unique weak solution $u$ of
problem \eqref{NSC}-\eqref{diri} belongs to $W^{2,2}(\O)$. Moreover,
there is a constant $C$ such that
\begin{equation}
 \|\,u\,\|_{2,2}\leq C\left(\|f\|+\|f\|_{\frac{6}{p+1}}^\frac{1}{p-1}\right)\,.\label{dn}
 \end{equation}
If $\,\O\,$ is convex (or the cubic domain $Q$) the result holds for
any $\,1\,<p\,\leq\, 2\,.$
\end{theorem}
It is worth noting that in the limit case $p=2$, when system
\eqref{NSC} reduces to the Poisson equations, we recover the well
known result
$$ \|\,u\,\|_{2,2}\leq C\,\|f\|\,.$$

 \vskip0.1cm We set
\be\label{c8} C_6=\max\{C_4, C_5\}\,,\ee and
\be\label{rq}r(q)=\left\{\begin{array}{ll}\dy
 \frac{3q}{3-(3-q)(2-p)}
&\dy \mbox{ if }\ q<3\,,\\
\hskip1cm q & \dy \mbox{ if }\ q> 3\,.
 \end{array}\right .\ee

\begin{theorem}\label{teoremaq}
Let be $\mu\geq 0$, $\,q>2\,$, and $\,1<\,p\leq 2\,$ such that
$\,(2-p)\,C_6<\,1\,$, where $C_6$ is given by \eqref{c8}. Let  $f\in
L^{r(q)}(\O)$ and let $u$ be the unique weak solution of problem
\eqref{NSC}--\eqref{diri}. Then $u$ belongs to $W^{2,q}(\O)$.
Moreover, the following estimate holds
\begin{equation}
 \|u\|_{2,q}\leq C
 \,\left(\|f\|_q+\|f\|_{r(q)}^\frac{1}{p-1}\right)\,.
\label{dnq}
\end{equation}
\end{theorem}

\vskip0.1cm
\begin{corollary}\label{corollaryq2}
Let  $p$, $\mu$ and $f$ be as in Theorem \ref{teoremaq}. Then, if
$q>3$, the weak solution of problem \eqref{NSC}--\eqref{diri}
belongs to $ C^{1,\alpha}(\overline\O)$, for $\alpha= 1-\frac 3q$.
\end{corollary}

Note that, in \eqref{rq}, $r(q)>q$ for any $q<3$. It is worth noting
that $r(q)$ tends to the same value $3$ as $q$ tends to $3\,,$ from
below and from above. Furthermore, if $q=2\,,$ the estimate
\eqref{dnq} becomes simply \eqref{dn}. Finally, in estimates
\eqref{dn} and \eqref{dnq}, the terms $\|f\|$ and $\|f\|_q$ can be
replaced by $1$.
\begin{remark} One could also consider the case where $f\in L^3(\Omega)$. We omit this
further case and leave it to the interested reader. In this regard
we stress that our interest mostly concerns the maximal
integrability of the second derivatives of the solution.
\end{remark}

\begin{remark}
When $p<2$ we could extend  to system \eqref{NSCS} the regularity
results up to the boundary obtained for system \eqref{NSC}, by
requiring a smallness condition on a suitable norm of $f$. Actually,
following arguments already used in \cite{crigri} and \cite{crigri2}
for non-Newtonian fluids, the idea is to study the regularity for
solutions of suitable approximating linear problems and then prove
the regularity for solutions of the nonlinear problem, by employing
the method of successive approximations. For brevity, here we avoid
this further development.
\end{remark}
\section{The $W^{2,2}(Q)$-regularity: $p>2$ and $\mu>0$}
In this section we prove Theorem \ref{main}. Therefore, throughout
the section we work in the cubic domain $Q$. Let us introduce the
definition of weak solutions of both the problems \eqref{NSC} and
\eqref{NSCS}.
\begin{definition}\label{noita}
Assume that $f \in V_p'(Q)$. We say that $u$ is a {\rm{weak
solution}} of problem \eqref{NSC}--\eqref{bc},
 if $\,u\,\in\, V_p(Q)$ satisfies
\begin{equation}\label{buf2a}
\int_{Q}\ S(\nabla \,u) \cdot \nabla\, \vp \,dx\ =\,\int_{Q}\, f\,
\cdot\,  \vp \,dx\,,
\end{equation}
for all $\vp \in \,V_p(Q)$.
\end{definition}

\begin{definition}\label{noit}
Assume that $f \in V_p'(\Omega)$. We say that $u$ is a {\rm{weak
solution}} of problem \eqref{NSCS}--\eqref{bc}
 if $\,u\,\in\, V_p(\O)$ satisfies
\begin{equation}\label{buf2}
\int_{Q}\ S(\cD \,u) \cdot \cD\, \vp \,dx\ =\,\int_{Q}\, f\, \cdot\,
\vp \,dx\,,
\end{equation}
for all $\vp \in \,V_p(Q)$.
\end{definition}

\par We recall that the existence and uniqueness of a weak solution
can be obtained by appealing to the theory of monotone operators,
following J.-L. Lions \cite{lions}.
\par In proving Theorem \ref{main} we focus on the symmetric
gradient case, since the full gradient case is, in some respects,
easier to handle. Hence we assume that $S$ is given by
$$S(\cD\,u)=\,(\,\m+|\,\cD\,u\,|)^{p-2}\,\cD\,u\,,$$ with $\mu>0$
and $p>2$.
\par
We follow arguments used in \cite{D}, in the context of
non-Newtonian fluids. Therefore, we will try to preserve the
notations. However in \cite{D} (due to the divergence free
constraint) the symbol $D^2_*u$ has a slightly different meaning
from that introduced in definition \eqref{SDstar} below, since it
also includes the derivatives $\partial_{33}^2\,u_3$ (see (2.8) in
\cite{D}).
\par As in in \cite{D}, in order to avoid arguments already
developed in other papers by the authors, we replace the use of
difference quotients simply by differentiation.
\par It is an easy matter to obtain the following Korn's type inequality, proceeding, for instance, as in
the proof given in \cite{pares}.
\begin{lemma}
There exists a constant $C$ such that
$$\|\,u\,\|_p\,+\,\|\,\nabla u\,\|_p\,\leq \,C\,\|\,\cD\,
u\,\|_p\,,$$%
for all $\,u \in V_{p}(Q)$. \label{TKorn}
\end{lemma}
\begin{lemma}
There exists a constant $C$ such that
$$\|\,D^2_*\,u\,\|_p\leq C\,\|\, \nabla_*\cD \, u\|_p \,,$$
for all $u\in V_p(Q)\,$. \label{l3.1}
\end{lemma}
This result reproduces Lemma 3.1 in \cite{D}, adapted to the new
definition of $D^2_*\,u\,$. Note that $\pa_s\,u=0$ on $\Gamma$,
$s=1,2$. \par Actually, the above two lemmas hold for each $p>1$.
\vskip0.2cm  Define, for $s = 1, 2$, \be\label{3.6} J_s(u)\,
:=\,\int_{Q} \nabla\cdot\left[\left(\mu+|\cD\,
u\,|\right)^{p-2}\cD\, u\right]\cdot \pa_{ss}^2u\, dx\,, \ee and
\be\label{3.10} I_s(u)\, :=\,\int_{Q} \left(\mu+|\cD\,
u\,|\right)^{p-2}\,|\pa_s\cD\, u\,|^2\, dx\,.\ee

\begin{lemma}\label{IJ} For any smooth function $u\in V_p(Q)$
the following inequality holds true \be\label{FI}J_s(u)\geq
C_1\,I_s(u)\,,\ee with the constant $C_1$ given by \eqref{tensas}.
 \end{lemma}
 \begin{proof}
 Integrating twice by
parts in \eqref{3.6}  one gets $$J_s(u)=\,\int_{Q}
\pa_s\left[\left(\mu+|\cD\, u\,|\right)^{p-2}\cD\, u\right]\cdot\,
\pa_s \nabla\,u\,dx\,.$$ Note that, due to symmetry, we replace
$\pa_s\, \nabla\,u$ by $\pa_s\,\cD\,u$. From the above expression,
one has
$$J_s(u)=
\,\int_{Q} \frac{\partial}{\partial
D_{kl}}\,\left[(\m+\,|D|)^{p-2}\,D_{ij}\right] \frac{\pa (\cD
u)_{kl}}{\pa x_s}\frac{\pa (\cD u)_{ij}}{\pa x_s} \,dx\,,
 $$ where the derivatives
with respect to $\,D_{k\,l}\,$ are evaluated at the point $D=\,\cD\,
u$. Note that here we merely appeal to the chain rule. Then the
result follows by using estimate \eqref{tensas}.
\end{proof}

Next we prove the following result which, roughly speaking, shows
that the second tangential derivatives of $u$ are square integrable.
\begin{lemma}\label{tan}
Assume that $f\,\in\,L^2(Q)$ and let $u$ be the solution of problem
\eqref{NSCS}--\eqref{bc}. Then $D^2_*u\,\in L^2(Q)$ and
\be\label{fe1} \|D^2_*u\,\|\leq \frac{
c}{\mu^{\,p-2}}\,\|\,f\,\|\,.\ee
\end{lemma}
\begin{proof}
Multiply both sides of the equations \eqref{NSC} by $\pa_{ss}^2\,u$,
$s=1,2$, and integrate over $Q$. By appealing to \eqref{3.6} and
Lemma \ref{IJ} it  readily follows that
$$I_s(u)\leq\, c\,\|\,f\,\|\, \|\,\pa_{ss}^2\,u\,\|\leq c
\,\|\,f\,\|\, \|\,\nabla\,\pa_{s}\,u\,\|\,,$$ hence, from Lemma
\ref{TKorn} applied to $\pa_s\,u$,
$$I_s(u)\leq c\,\|\,f\,\|\,\|\,\pa_{s}\,\cD\,u\,\|\,.$$
Finally, observing that
$$\mu^{p-2}\,\|\,\pa_{s}\,\cD\,u\,\|^2\leq I_s(u)\,,$$
one gets
$$\|\,\pa_{s}\,\cD\,u\,\|\leq\, \frac{c}{\mu}{\atop^{p-2}}\,\|\,f\,\|\,.$$
Application of Lemma \ref{l3.1}, gives the result.
\end{proof}
In order to complete the proof of Theorem \ref{main} we have to show
the integrability of the remaining second derivatives, namely the
normal derivatives $\pa_{33}^2\,u$. In doing this we follow the
argument used in the paper \cite{bvlali}; we express these
derivatives, pointwisely, in terms of the derivatives of $u$ already
estimated, and solve the corresponding system in the unknowns
$\pa_{33}^2\,u_i$, $i=1,2,3$. Note that the main differences between
this situation and that in reference \cite{bvlali}, are the
following: in \cite{bvlali} the $L^2$- integrability of
$\pa_{33}^2\,u_3$ is known thanks to the divergence free constraint,
$\pa_{33}^2\,u_3=-\pa_{31}^2\,u_1-\pa_{32}^2\,u_2$. Hence the
$3\,\times\,3$ linear system considered below is replaced, in
\cite{bvlali}, by a $2\,\times\,2$ linear system in the unknowns
$\pa_{33}^2\,u_i$, $i=1,2$. On the other hand, in reference
\cite{bvlali}, the presence of the pressure prevents the full
$W^{2,2}$-regularity.\par%
For the missing derivatives we prove the following lemma.
\begin{lemma}\label{mis}
Let $u$ be the solution of problem \eqref{NSCS}--\eqref{bc}. Then
the  vector field  $\partial_{33}^2\,u$ satisfies the pointwise
estimate \be
|\,\pa_{33}^2u\,|\leq\,c\,\left(\,\frac{1}{\mu}{\atop^{p-2}}\,|f\,|+
\, |\,D^2_*u\,|\right)\,,\ \mbox{a.e. in } Q\,.\ee\label{fe}
\end{lemma}
\begin{proof}
Straightforward calculations show that
\be\label{315}\begin{array}{ll}\vspace{1ex}\dy \pa_s\,[(\mu+|\cD\,
u|)^{p-2}\,\cD\, u]=(\mu+|\cD\, u|)^{p-2}\,\pa_s\,\cD\, u\\\hskip3cm
\dy + \,(p-2)(\mu+|\cD\, u|)^{p-3}\,|\cD\, u|^{-1}\left(\cD\, u\cdot
\pa_s\,\cD\, u\right)\cD\, u\,.\end{array}\ee
 For convenience, we
set $\cD_{jk}=(\cD\, u)_{jk}$ and $B\,:=\,(\mu+|\cD \,u|)$. By using
\eqref{315}, the $j^{th}$ equation \eqref{NSC}, for any $j=1,2,3$,
  takes the following form
  \be
 B^{p-2}\left(\pa^2_{kk}u_j+\pa^2_{jk}u_k\right)
+(p-2)B^{p-3}|\cD\,
 u|^{-1}\cD_{lm}\cD_{jk}\left(\pa^2_{km}u_l+\pa^2_{kl}u_m\right)
 =-2f_j.\label{S1}\ee
Let us write the previous three equations as a system in the
unknowns $\pa_{33}^2\,u_j$. For $j=1,2$ we have

 \be\label{S2}
\begin{array}{ll}\vspace{1ex}\dy
B^{p-2}\,\pa^2_{33}\,u_j+2(p-2)\,B^{p-3}\,|\cD\,
 u|^{-1}\,\cD_{j3}\,\sum_{l=1}^3\,\cD
 _{l3}\,\pa^2_{33}\,u_l =\,F_j\,-2\,f_j\,,
\end{array}\ee
where \be\label{S3}
\begin{array}{ll}\vspace{1ex}\dy
F_j:=&\displaystyle-\,B^{p-2}\,\sum_{k=1}^2\,\pa^2_{kk}\,u_j-B^{p-2}\,\sum_{k=1}^3\,\pa^2_{jk}\,u_k
\\&\displaystyle -2(p-2)\,B^{p-3}\,|\cD\,
 u|^{-1}
\!\!\sum_{l,m,k=1\atop (m,k)\not=(3,3)
}^3\pa^2_{km}\,u_l\,\cD_{jk}\,\cD_{lm}\,.
\end{array}\ee
For $j=3$ we have
 \be\label{S4}
\begin{array}{ll}\vspace{1ex}\dy
2B^{p-2}\,\pa^2_{33}\,u_j+2(p-2)\,B^{p-3}\,|\cD\,
 u|^{-1}\,\cD_{j3}\,\sum_{l=1}^3\,\cD
 _{l3}\,\pa^2_{33}\,u_l =\,F_j\,-2\,f_j\,,
\end{array}\ee
where, for $j=3$,  \be\label{S5}
\begin{array}{ll}\vspace{1ex}\dy
F_j:=&\displaystyle-
B^{p-2}\,\sum_{k=1}^2\,\pa^2_{kk}\,u_j-\,B^{p-2}\,\sum_{k=1}^2\,\pa^2_{jk}\,u_k
\\&\displaystyle -2(p-2)\,B^{p-3}\,|\cD\,
 u|^{-1}
\!\!\sum_{l,m,k=1\atop (m,k)\not=(3,3)
}^3\pa^2_{km}\,u_l\,\cD_{jk}\,\cD_{lm}\,.
\end{array}\ee
The equations \eqref{S2}, for $j=1,2$, together with the equation
\eqref{S5} for $j=3$ can be treated as a $3\times 3$ linear system
in the unknowns $\pa^2_{33}u_j$, $j=1,2,3$. Multiply all three
equations by $B^{2-p}$. We denote the elements of the matrix
$A=A(x)$ associated with this system as $a_{jl}$, where $j,l=1,2,3$.
Then, we can write the system in a compact form as
\be\label{sistema1} a_{jl}\,\pa^2_{33}\,u_l=G_j\,,\ee where the
elements of the matrix of the system are given by
$$a_{jl}:=\delta_{jl}+2(p-2)\,\left(B\,|\cD\, u|\right)^{-1}\cD_{j3}\cD_{l3}\,,$$
 for $j=1,2$,  by
$$a_{jl}:=2\delta_{jl}+2(p-2)\,\left(B\,|\cD\, u|\right)^{-1}\cD_{j3}\cD_{l3}\,,$$
and for $j=3$, and \be\label{G}
 G_j:=B^{2-p}\left(\,F_j\,-2\,f_j\,\right)\,.\ee
 Note that
$a_{jl}=a_{lj}$; moreover, if $\xi$ denotes any vector field then
$$a_{jl}\xi_j\xi_l=|\,\xi\,|^2+\xi_3^2+2\,(p-2)\,\left(\,B\,|\cD\,
u|\,\right)^{-1} [\,{\mathcal{D}}\, u\cdot\xi\,]_{3}^2\,.$$ Hence,
the matrix $A=(a_{jl})$ is also definite positive, a.e. in $x\in Q$,
and the previous identity shows that
$$a_{jl}\,\xi_j\,\xi_l\geq\,|\,\xi\,|^2\,.$$ By setting
$\xi=\pa_{33}^2\,u$, we have obtained \be\label{ovv1}
|\,\pa_{33}^2\,u\,|^2\leq \,|\,G\,|\,|\,\pa_{33}^2\,u\,|,\
\mbox{a.e. in } Q\,,\ee where, obviously, by $G$ we mean the vector
$(G_1,G_2,G_3)$. Noting that, from \eqref{G}, \eqref{S3} and
\eqref{S5}, there holds \be\label{ovv2} |\,G_j\,|\leq
\frac{2}{\mu}{\atop^{p-2}}\,|f_j\,|+ c\, |\,D^2_*u\,|\,,\ \mbox{a.e.
in } Q\,,\ee from this estimate and \eqref{ovv1} we get \eqref{fe}.
\end{proof}
Finally, by combining \eqref{fe1} and \eqref{fe} we readily obtain
$$\|D^2\,u\,\|\leq \frac{c}{\mu}{\atop^{p-2}}\,\|\,f\,\|\,$$
which is just \eqref{estsdp2}. The proof of Theorem \ref{main} is
accomplished.

\section{A regularity result for an approximating system: $p<2\,$.}
In the sequel we introduce an auxiliary positive parameter $\eta$
and study the regularity for solutions of the following
approximating problem
\begin{equation}\left\{
\begin{array}{ll}\vspace{1ex}
-\eta\,\Delta v-\nabla \cdot S\,(\nabla\,v) =\,f,\ \mbox{ in } \O\,,
\\%
v=\,0,\ \mbox{ on } \partial \O\,,
\end{array}\right .
\label{App}
\end{equation}
with $S$ defined by \eqref{tensord},  $\eta>0$, $\mu>0$ and
$p\in(1,2)$. The solutions $v_\eta$ satisfy the estimate
\eqref{mainest1} below, with the constant $C$ independent of $\eta$.
This allows us to show that, as $\eta\to 0$, $v_\eta$ tends, in a
suitable sense, to the solution  $v$ of problem \eqref{App} with
$\eta=0$. A similar situation occurs, with respect to $\mu$, as
$\mu\to 0$.
\par We explicitly note that we introduce the above model just to
approximate our solution by smooth functions.
\par Let us introduce the definition of weak solution of
 both the problems \eqref{App} and \eqref{NSC}--\eqref{diri}.
\begin{definition}\label{wsapp}
Assume that $f \in V_2'(\O)$. We say that $v$ is a {\rm{weak
solution}} of problem \eqref{App} if $v\in V_2(\O)$ and satisfies
\be\label{buf3} \eta\int_{\O}\nabla v\cdot \nabla \varphi\,
dx+\int_{\Omega}\, S(\nabla\,v) \cdot {\nabla} \varphi \,dx
=\,\int_{\Omega}\ f \cdot \varphi \,dx\,, \ee for all $\varphi \in
V_2(\O)$.
\end{definition}

\begin{definition}\label{noitaa}
Assume that $f \in V_p'(\O)$. We say that $u$ is a {\rm{weak
solution}} of problem \eqref{NSC}--\eqref{diri},
 if $\,u\,\in\, V_p(\O)$ satisfies
\begin{equation}\label{buf2aa}
\int_{\O}\ S(\nabla \,u) \cdot \nabla\, \vp \,dx\ =\,\int_{\O}\, f\,
\cdot\,  \vp \,dx\,,
\end{equation}
for all $\vp \in \,V_p(\O)$.
\end{definition}
\par As recalled in the previous section, the existence and uniqueness of a weak
solution is known from the theory of monotone operators.
\par
We start by proving the $W^{2,2}$-regularity result stated in
Proposition \ref{Vteo1} below. In \eqref{mod222}, the dependence of
the constant $c$ on $\Omega_0$, $\eta$ and $\mu$ is omitted since
the aim of the proposition is just to ensure that second derivatives
are well defined a.e. in $\Omega$. Following the notations
introduced in section 2, by capital letters, $C$, $C_1$, $C_2$,
etc., we denote positive constants independent of $\mu$ and of
$\eta$ also.

\begin{proposition}\label{Vteo1}
Let $p\in\left (1,2\right)$, $f \in L^{2}(\Omega)$, and $\,v$ be a
weak solution of problem \eqref{App}. Then $\,v\,\in
W^{2,2}_{loc}(\O)$ and, for any fixed open set
$\Omega_0\subset\subset \O$, there exists a constant $c$ such that
\begin{equation}\label{mod222}\|\,D^2 v\,\|_{L^2(\O_0)}\leq c\,\|f\|\,.\end{equation}
\end{proposition}

\begin{proof} As in the previous section, we formally use
derivatives instead of difference quotients, to make the computation
simpler.  Fix an open set $\Omega_0\subset\subset \O$. Let $\zeta$
be a $C_0^2(\Omega)$-function, such that $0\leq\,\zeta(x) \leq\,1$
in $\Omega$, and $\zeta(x)=\,1$ in $\Omega_0$. Multiplying the first
three equations in \eqref{App} by $-\,\nabla \cdot\,(\zeta^2\,\nabla
\,v)\,$ and integrating over $\Omega$ we get
\begin{equation}\label{gesseag}
\begin{array}{ll}\displaystyle\vspace{1ex}
\eta\,\int_{\Omega} \partial^2_{jj}\,v_i\,
\partial_h\left(\zeta^2 \,\partial_h\,v_i\right)\,dx+ \int_{\Omega} \partial_j
\left[(\mu+\,|\nabla \,v|)^{p-2}\,(\nabla\, v)_{i\,j}\right]\,
\partial_h\left(\zeta^2 \,\partial_h\,v_i\right)\,dx\\
\displaystyle\hskip 1cm
=-\,\int_{\Omega}f_i\,\partial_h\,\left(\zeta^2
\,\partial_h\,v_i\right)\,dx\,.
\end{array}\end{equation}
By integration by parts, with respect to $x_j$ and $x_h$, on the
left-hand side one has
\begin{equation}\label{gesse}
\begin{array}{ll}\displaystyle\vspace{1ex}
\eta\,\int_{\Omega} (\partial^2_{jh}\,v_i)^2\, \zeta^2 \,dx
+\,\int_{\Omega}
\partial_h\,\left[(\mu+\,|\nabla v|)^{p-2}\,(\nabla
v)_{i\,j}\right]\,\partial_h (\nabla v)_{i\,j}\,\zeta^2\,dx
\\ \displaystyle\hfill
=-\eta\,\int_{\Omega}
\left(\partial^2_{jh}\,v_i\,\right)\,R_{i\,j\,h}(x) \,dx-
\int_{\Omega}
\partial_h\,\left[(\m+\,|\nabla v|)^{p-2}\,(\nabla v)_{i\,j}\right]
\,R_{i\,j\,h}(x)\,dx\\ \displaystyle\hfill -\,\int_{\Omega}f_i
\left(\,
\partial^2_{hh}\,v_i\,\right)\zeta^2\,dx-2\,\int_{\Omega}f_i \left(\,
\partial_h\,v_i\right)\,\zeta\,\left(\pa_h\,\zeta\right)\,dx=\,\sum_{i=1}^4 I_i\,,
\end{array}
\end{equation}
where, with obvious notation, $R_{i\,j\,h}$ are lower order terms
satisfying estimates
\begin{equation}\label{erros}
|R_{i\,j\,h}(x)| \leq\,c\,|\zeta|\,|\nabla \,\zeta|\,|\nabla v|\,.
\end{equation}
As in the proof of Lemma \ref{IJ}, it is easy to verify, by
appealing to \eqref{tensas}, that

\begin{equation}\label{jess3}
\,\int_{\Omega}
\partial_h\,\left[(\mu+\,|\nabla v|)^{p-2}\,(\nabla
v)_{i\,j}\right]\,\partial_h (\nabla v)_{i\,j}\,\zeta^2\,dx\geq\,c
\,\,\int_{\Omega} \,\left(\mu+\,|\nabla v|\right)^{p-2}\,|D^2
v|^2\,\zeta^2\,dx\,.
\end{equation}On the other hand, by H\"{o}lder's and Cauchy-Schwartz inequalities,
\begin{equation}\label{bem3}
|I_1| \leq\, \,\epsilon \,\|\,|D^2
v|\,\zeta\|^2\,+\,c(\epsilon)\,\|\nabla\,\zeta\|^2_{\infty}\,\|\nabla\,v\|^2\,,
\end{equation}
\begin{equation}\label{bem4}
|I_3| \leq\,
\,\epsilon\,\|\,|D^2v|\,\zeta\|^2\,+\,c(\epsilon)\,\|f\|^2\,,
\end{equation}
and
\begin{equation}\label{bem5}
|I_4| \leq\, c\,\|\nabla \,\zeta\|_{\infty}\|f\|\,\|\nabla \,v\|\,.
\end{equation}
Further, by using the estimate $$ \frac{\pa S_{i\,j}(A)}{\pa
A_{k\,l}}\leq\, c\,(\m+\,|\,A\,|)^{p-2},$$ we have
\begin{equation}\label{idois}
|I_2| \leq\,c\,\int_{\Omega} (\m+\,|\nabla v|)^{p-2}\,|D^2
\,v|\,|\zeta|\,|\nabla \,\zeta|\,|\nabla v|\,dx\,,
\end{equation}
and, by the Cauchy-Schwartz inequality,
\begin{equation}\label{bem}
|I_2| \leq\, \,\epsilon
\,\|\,|D^2v|\,\zeta\|^2+\,c(\epsilon)\,\|\nabla\,\zeta\|^2_{\infty}\,\|\nabla\,v\|^2_2\,.
\end{equation}
From \eqref{gesse} together with $\|\nabla\,v\|\leq c\,\|f\|$, it
follows that
\begin{equation}\label{ies2}
\begin{array}{ll}\vspace{1ex}\dy
\|\,|D^2v|\,\zeta\|\leq c\,\|f\|\,.
\end{array}
\end{equation}
Hence \eqref{mod222} holds.
\end{proof}
\par
Our next step is to get a global estimate for the $L^2$-norm of the
second derivatives, uniform in $\eta$. This is the aim of
 the following proposition.
\begin{proposition}\label{Vteo2}
Let be $\,(2-p)C_4<\,1$, with $C_4$ given by \eqref{lad2}. Let $f
\in L^{\frac{6}{p+1}}(\Omega)$, and let $v$ be a weak solution of
problem \eqref{App}. Then  $v$ belongs to $W^{2,2}(\O)$. Moreover,
there exists a constant  $C$ such that
\begin{equation}\label{mainest1}\|\,v\,\|_{2,2}
\leq C\left(\|f\|+\|f\|_{\frac{6}{p+1}}^{\frac
{1}{p-1}}\right)\,.\end{equation}
\end{proposition}
\begin{proof} In order to avoid a useless dependence on $\mu$, we assume, without loss of generality,
$\mu\in(0,1]$. At first note that, by replacing $\varphi$ by $v$ in
\eqref{buf3} it is easy to get the following estimate for $\|\nabla
v\|_p$, uniformly in $\eta$,
 $$ \|\nabla v\|_{p}^{p}\leq
 \mu^p\,|\O|+\,2^{2-p}\int_\O f\cdot v\,dx\, \leq C\,\left(1+\int_\O f\cdot v\,dx\,\right)\,.$$
Since, by Proposition \ref{Vteo1}, $v\in W^{2,2}_{loc}(\O)$, the
$i^{\rm th}$ equation \eqref{App} can be written almost everywhere
in $\O$ as \be\label{eqbis}\ba{ll}\dy\vs1 \eta \Delta
v_i+\,(\m+\,|\nabla v|)^{p-2}\,\Delta
v_i\\\hskip2cm\dy+(p-2)(\m+\,|\nabla v|)^{p-3}\,|\nabla
v|^{-1}\nabla v\cdot \left({\pa_j}\,\nabla v\right) \pa_j\,v_i
=-\,f_i\,.\ea\ee By multiplying both sides by $\,\Delta v_i\,$ and
summing over $i=1,2,3$, we have
$$\ba{ll}\vs1 \dy\eta\, |\Delta
v|^2+\,(\m+\,|\nabla v|)^{p-2}\,|\Delta v|^2\\
\dy\hfill=(2-p)(\m+\,|\nabla v|)^{p-3}|\nabla v|^{-1}\nabla v\cdot
\left( \pa_j\,\nabla v\right)\pa_j\,v_i\, \Delta v_i -\,f_i\, \Delta
v_i\,, \mbox{ a.e. in } \Omega\,.\ea$$ Next, we drop the term $\eta
|\Delta v|^2$, and bound the left-hand side from below by
$\,(\,\m+\,|\nabla v|\,)^{\,p-2}\,|\Delta v|^2$. Multiplying the
estimate thus obtained by $(\m+\,|\nabla v|)^{2-p}$ and then
integrating over $\O$ we get
$$ \int_{\O}|\,\Delta
v\,|^2\,dx\leq (2-p)\int_{\O}|\,D^2 v\,|\,|\,\Delta
v\,|\,dx+\int_{\O}(\,\m+\,|\nabla v|\,)^{2-p}\,|\,f\,|\,|\,\Delta
v\,|\,dx\,,
$$ where we have used the estimate (for details see the Appendix)
 $$\left|\,\nabla v\cdot
\left( \pa_j\,\nabla v\right)( \pa_j\,v_i\,)\, \Delta
v_i\,\right|\leq \,|\,\nabla v\,|^2\, |\,D^2\,v\,|\,|\,\Delta
v\,|\,.$$ Observing that $(\m+\,|\nabla v|)^{2-p}\leq
\m^{2-p}+\,|\nabla v|^{2-p}$, using H\"{o}lder's inequality, and
dividing both sides by $\|\,\Delta v\,\|$, we get \be\label{lap1}
\|\,\Delta v\,\|\leq (2-p)\,\|\,D^2v\,\|+\|\,|\nabla
v|^{2-p}\,f\,\|+\,\|\,f\,\|\,. \ee Let us estimate the first two
terms on the right-hand side. For the first term we employ estimate
\eqref{lad2}.
 As far as the second  term in \eqref{lap1} is
concerned, by applying H\"{o}lder's inequality with exponents
$3/(2-p)$ and $3/(p+1)$, the Sobolev embedding of $W^{2,2}(\O)$ in $
W^{1,6}(\O)$, and by appealing to the estimate \eqref{lad1q} with
$q=2$, we get
$$\ba{ll}\dy\vs1 \|\,|\nabla v|^{2-p}\,f\|\leq \|\,\nabla
v\,\|_6^{2-p}\, \|\,f\|_{\frac{6}{p+1}}\leq C\,\|\,\Delta
v\,\|^{2-p}\, \|\,f\|_{\frac{6}{p+1}}. \ea$$ By using the above
estimates  in \eqref{lap1}, we get
$$ \|\,\Delta v\,\|\leq (2-p)\,C_4\|\,\Delta v\,\|+ C\,\|\,\Delta
v\,\|^{2-p}\, \|\,f\|_{\frac{6}{p+1}}+\,\|\,f\,\|\,.$$ Recalling
that $(2-p) C_4<\, 1$, and applying the Young's inequality
\be\label{you} a^{2-p}\,b\leq \varepsilon \, a+\, c(\varepsilon)\,
b^\frac{1}{p-1}\,,\ee
 it is easy to recognize
that the estimate \be\label{rep42}\|\,\Delta v\,\|\leq
\,C\,\left(\|f\|+\|f\|_{\frac{6}{p+1}}^{\frac{1}{p-1}}\right)\ee
holds.
 By using once again \eqref{lad1q} we prove \eqref{mainest1}.
\end{proof}

\section{The $W^{2,2}$-regularity result: $p<2\,$.}

\begin{proof}[Proof of Theorem \ref{teorema}]
We deal separately with the case $\mu>0$ and the degenerate case
 $\mu=0$. \vskip0.2cm \emph{ The case $\mu>0$ } -
Consider the ``sequence'' $(v_\eta)$ consisting of the solutions to
problem \eqref{App}, for $\eta>0$. By the above proposition the
sequence $(v_\eta)$ is uniformly bounded in $W^{2,2}(\O)$.
Therefore, by Rellich's theorem, there exists a field $u\in
W^{2,2}(\O)$ and a subsequence, which we continue to denote by
$(v_\eta)$, such that $v_\eta\rightharpoonup u$ weakly in
$W^{2,2}(\O)$, and strongly in $W^{1,q}(\O)$ for any $q<6$. Let us
prove that \be\label{a2}
 \int_{\Omega}S(\nabla u) \cdot \nabla
\varphi \,dx= \lim_{\eta\to 0^+}\left\{\int_{\Omega}S(\nabla v_\eta)
\cdot \nabla\varphi \,dx+\eta\int_{\Omega}\nabla v_\eta\cdot
\nabla\varphi \,dx\right\}\,,\ee for any $\varphi \in
C_0^{\infty}(\Omega)$. By applying \eqref{tensorS1} and then
H\"{o}lder's inequality, we get
$$\ba{ll}\vs1\dy\left|\int_{\Omega} S(\nabla u)
\cdot \nabla \varphi \,dx-\int_{\Omega}S(\nabla v_\eta) \cdot
\nabla\varphi \,dx\right|\\
\vs1\dy\leq c\,\int_{\Omega}\! \,\left(\,\mu+|\nabla u\,|+|\nabla
v_\eta|\,\right)^{p-2}\,|\,\nabla u-\nabla v_\eta\,|\,|\nabla
 \varphi|\, dx\,\\\vs1\dy
  \leq c\,\int_{\Omega}\! \,\left|\,\nabla u-\nabla
v_\eta\,\right|^{p-1}\,|\nabla
 \varphi|\, dx\,\leq c\,\|\,\nabla
v_\eta\!-\nabla u\,\|_p^{p-1}\,\|\,\nabla \varphi\,\|_p\,.\ea$$ The
right-hand side of the last inequality tends to zero, as $\eta$ goes
to zero, thanks to the strong convergence of $v_\eta$ to $u$ in
$W^{1,p}(\O)$. Further
$$\left|\,\eta\int_{\Omega}\nabla v_\eta\cdot \nabla\varphi
\,dx\right|\leq \eta\, \|\nabla v_\eta\|\,\|\nabla \vp\|\,,
$$
where the right-hand side tends to zero as $\eta$ goes to zero.
Finally, observing that for any $\eta>0$ and any $\vp\in
C_0^{\infty}(\Omega)$ the right-hand side of \eqref{a2} is equal to
$\int_\O f\cdot \vp\,dx$, we show that $u$  satisfies the integral
identity \eqref{buf2aa} for any $\vp\in C_0^{\infty}(\Omega)$. By a
standard argument we show that $u$ satisfies the integral equation
\eqref{buf2aa}, for any $\vp\in V_p(\O)$. Hence $u$ is a weak
solution of \eqref{NSC}, and belongs to $\,W^{2,2}(\O)$. Moreover,
\eqref{dn} follows from the relation $\|\,\,u\,\|_{2,2}\leq\dy
\liminf_{\eta\to 0^+} \|\,v_\eta\,\|_{2,2}$, together with
\eqref{mainest1}. From the uniqueness of weak solutions we obtain
the desired result. \vskip0.2cm \emph{ The case $\mu=0$} - Let us
denote by $u_\mu$ the sequence of solutions of \eqref{NSC} for the
different values of $\mu>0$. We have shown that the sequence
$(u_\mu)$ is uniformly bounded in $W^{2,2}(\O)$. Therefore, exactly
as above, we can prove the weak convergence of a suitable
subsequence in $W^{2,2}(\O)$, and the strong convergence in
$W^{1,q}(\O)$ for any $q<6$, to the solution $u\in W^{2,2}(\O)$ of
the problem \eqref{NSC} with $\mu=0$. In this regard note that
estimate \eqref{tensorS1} also holds with $\mu=0$.

\vskip0.2cm  Finally we prove the last assertion in Theorem
\ref{teorema}. For a smooth convex domain $\O$ estimate \eqref{lad2}
holds with $C_4=1$. Hence the assumption on $p$ is merely $p>1$.
\end{proof}
\begin{remark}\label{rempmag2}
We could adapt the above arguments to the case $p>2$. Via a result
similar to Proposition \ref{Vteo1}, one shows that the solution $v$
of the approximated system \ref{App} belongs to $W^{2,2}_{loc}(\O)$.
Then reasoning as in the proof of Proposition \ref{Vteo2}, one
obtains a global estimate for $v$ in $W^{2,2}(\O)$, uniformly in
$\eta$, with a restriction on the range of $p$, $p\in (2,2+\frac
{1}{C_4})$, $C_4$ as in \eqref{lad2}. Hence, as  Theorem
\ref{teorema} above, one proves that the solution of \eqref{NSC},
with $\mu>0$, belongs to $W^{2,2}(\O)$. This result has the
advantage  to be directly proved in a general smooth domain, without
need of localization techniques. However, it requires limitations on
the range of $p$ and, moreover, it cannot directly cover the case
$\mu=0$, since the $W^{2,2}(\O)$-estimates that one obtain are not
uniform in $\mu$.
\end{remark}

\section{The $W^{2,q}$-regularity result: $q\geq\, 2$ and $p<2\,$.}
\begin{proof}[Proof of Theorem \ref{teoremaq}]
From Theorem \ref{teorema} we already know that the solution $u$ of
problem \eqref{NSC} belongs to $W^{2,2}(\Omega)$, since
$(2-p)C_4<\,1\,$. Therefore, we can write equation \eqref{eqbis}
with $u$ in place of $v$, and $\eta=0$. By multiplying this equation
 by $\left(\mu+|\nabla\,u|\right)^{2-p}$, we can write, a.e. in
$\O$, \be\label{reg1}-\Delta u-(p-2) \frac{\nabla u\cdot\nabla\nabla
u\cdot \nabla u}{(\mu+\,|\nabla u|)\,|\nabla u|}=
f\left(\mu+|\nabla\,u|\right)^{2-p}\,,\ee where we have used the
notation $\nabla u\cdot\nabla\nabla u\cdot \nabla u$ to denote the
vector whose $i^{\rm th}$ component is $\nabla u\cdot \left(
\pa_j\,\nabla u\right) \pa_j\,u_i=(\pa_l\,u_k)
\,(\pa^2_{j\,l}\,\,u_k)\,(\pa_j\,u_i)\,.$

\par
We start by proving an a priori $L^q$-estimate for the second
derivatives of $u$ by assuming, for the moment, that $u\in
W^{2,q}(\O)$. We follow an argument similar to that used for proving
the $W^{2,2}$-estimates of $u$. We multiply both sides of equation
\eqref{reg1} by $-\,\Delta u\,|\Delta u|^{q-2}$, and integrate in
$\Omega$. We get (for details see the Appendix)
$$\int_{\O}|\,\Delta u\,|^q\,dx\\ \dy \hfill\leq (2-p)\int_{\O}|\,D^2
u\,|\,|\,\Delta u\,|^{q-1}\,dx+\int_{\O}(\,\m+\,|\nabla
u|\,)^{2-p}\,|\,f\,|\,|\,\Delta u\,|^{q-1}\,dx\,. $$%
By appealing to H\"{o}lder's inequality and to the inequality
$(\m+\,|\nabla u|)^{2-p}\leq 1+\,|\nabla u|^{2-p}$, we show that
\be\label{lap1w}\ba{ll}\vs1\dy \|\,\Delta u\,\|_q^q\,\leq
&\dy\!\!(2-p)\,\|\,D^2u\,\|_q \|\,\Delta u\,\|_q^{q-1}\\\hfill&\dy
+\,\|\,f\,\|_q\|\,\Delta u\,\|_q^{q-1}+ \|\,|\nabla
u|^{2-p}\,f\,\|_q\|\,\Delta u\,\|_q^{q-1}\,.\ea \ee Further, by
dividing both sides by $\|\,\Delta u\,\|_q^{q-1}$, one gets
\be\label{lap122w} \|\,\Delta u\,\|_q\leq (2-p)\,\|\,D^2u\,\|_q
+\,\|\,f\,\|_q+\|\,|\nabla u|^{2-p}\,f\,\|_q. \ee We estimate the
first term  on the right-hand side of \eqref{lap122w} via inequality
\eqref{ladaq}.\par
 Concerning the last term on the right-hand side, we start by
assuming that $q\in \left(2,\,3\right)$. As usual we denote by
$q^*=3p/(3-p)$ the Sobolev embedding exponent of $q$. By applying
H\"{o}lder's inequality, with exponents $s=q^*/(2-p)q$ and
$s'=r(q)/q$, 
we get \be\label{cmha}\|\,|\nabla u|^{2-p}\,f\|_q\leq \dy\|\,\nabla
u\,\|_{q^*}^{2-p}\, \|\,f\|_{r(q)} \,.\ee From \eqref{lap122w}, by
appealing to \eqref{lad1q}, \eqref{cmha} and to Young's inequality
 one easily gets
$$
\| \Delta u\|_{q}\leq (2-p)\,C_5\| \Delta u\|_{q}+\|f\|_{q}+
\varepsilon\| \Delta u\|_{q}+c(\varepsilon)
 \|f\|_{r(q)}^\frac{1}{p-1}\,.
$$
Recalling the assumption on $p$, a further application of estimate
\eqref{lad1q} gives
\begin{equation}\label{dntipo}
 \| u\|_{\,2,q}\leq
 C\left(\|f\|_q+\,\|f\|_{r(q)}^\frac{1}{p-1}\right)\,.
\end{equation}
\par
Next we assume that $q>3$. We will use arguments similar to the
previous ones. Actually, by appealing to the Sobolev embedding
$W^{1,q}(\Omega) \hookrightarrow L^{\infty}(\Omega)$, to the
estimate \eqref{lad1q}, and to Young's inequality, we estimate the
last term on the right-hand side of \eqref{lap122w} as follows:
 \be\label{cmhaq}
\|\,|\nabla u|^{2-p}\,f\|_q\leq \dy\|\,\nabla u\,\|_{\infty}^{2-p}\,
\|\,f\|_{q}\leq\dy C\,\|\,\Delta u\,\|_{q}^{2-p}\, \|\,f\|_{q}\leq
\varepsilon\,\|\,\Delta u\,\|_{q}+c(\varepsilon)\,
\|\,f\|_{q}^{\frac{1}{p-1}} \,.\ee Then, by repeating verbatim the
arguments used above, one shows that
 $u$ is bounded in $W^{2,q}(\O)$, uniformly with respect to $\mu$, and that the  estimate
\eqref{dntipo} holds. Finally, the argument used in the proof of the
Theorem \ref{teorema} in order to extend the results to the
degenerate case $\mu=0$ apply here as well. \vskip0.2cm The previous
arguments are formal, since we have assumed that solutions belongs
to $W^{2,q}(\O)$. However the following argument applies. Let us
consider the problem
\be\label{rega1}\left\{\ba{ll}\vspace{1ex}\dy-\Delta w^{\ve}-(p-2)
\frac{\nabla J_\ve(u)\cdot\nabla\nabla w^\ve\cdot \nabla J_\ve(u)}
{\left(\mu+\,J_\ve(|\nabla u|)\right)\,J_\ve(|\nabla u|)}=
f\left(\mu+|\nabla\,u|\right)^{2-p},\quad\mbox{ in } \Omega\,,\\\dy
w^{\ve}=0\,, \quad \mbox{ on } \partial \O\,,\ea\right . \ee  where
$w^\ve$ is the unknown and $J_\ve$ denotes the Friedrichs mollifier.
The coefficients of this modified system belong to
$C^{\infty}(\R^n)$. We can also write this system in divergence
form, as follows: \be\label{divf}-
\pa_{h}\left[\,m_{ijhk}(x)\,\pa_k\,w^\ve_j\,\right]+(p-2)\,\pa_{h}
\left[\,c^{\ve}_{ijhk}(x)\,\right]\,\pa_{k}\,w^\ve_j =
f\left(\mu+|\nabla\,u|\right)^{2-p}\,,\ee where $$m_{ijhk}(x) =
\delta_{ij}\delta_{hk}+ (p-2)\,c^\ve_{ijhk}(x)$$ and
$$c^{\ve}_{ijhk}(x) =\pa_{h} J_\ve(u_i)\,\pa_{k}J_\ve(u_j)
\frac{1}{\left(\mu+\,J_\ve(|\nabla u|)\right)\,J_\ve(|\nabla u|)}.$$
Further, let
$$c_{ijhk}(x)
=(\,\pa_{h}\, u_i\,)\,(\,\pa_{k}\,u_j\,)
\frac{1}{\left(\mu+\,|\nabla u|\right)\,|\nabla u|}\,.$$ From the
well known estimate \be\label{a} |\nabla J_\ve(u)|= |J_\ve(\nabla
u)|\leq J^\ve(|\nabla u|)\ee we get
\be\label{b}|c^{\ve}_{ijhk}(x)|\leq 1,\ \ \mbox{ uniformly in } x\,,
\ve\,, \mbox{ and } \mu\,.\ee This shows that the system
\eqref{divf} (hence, the system \eqref{rega1}) is a linear elliptic
system with regular coefficients. For such a system it is well known
that if a force term $F$ belongs to $L^q(\O)$, $\,q\geq 2\,$, then
the solution belongs to $W^{2,q}(\O)$ (see, for instance,
\cite{GiaqMart}). By following the previous arguments with $u$
replaced by $w^\ve$, and by using \eqref{a} and \eqref{b}, it is
straightforward to obtain the estimate \eqref{dntipo} for $w^\ve$.
Note that such estimates are uniform with respect to $\mu$ and
$\ve$. Hence, there exists a subsequence, still denoted by $w^\ve$,
and an element $w\in W^{2,q}(\O)$ such that, as $\ve$ goes to zero,
$w^\ve$ converges to $w$, weakly in $W^{2,q}$. Convergence is also
strong in $W^{1,r}(\O)$: for any $r$ if $q> 3$, and for any $r\in
\left(1,\frac {3q}{3-q}\right)$ if $q<3$. Let us show that $w$ is a
solution of the system \be\label{rega1w}\dy-\Delta w-(p-2)
\frac{\nabla u\cdot\nabla\nabla w\cdot \nabla u}
{\left(\mu+\,|\nabla u|\right)\,|\nabla u|}=
f\left(\mu+|\nabla\,u|\right)^{2-p}. \ee To this purpose, we write
equations \eqref{rega1} and \eqref{rega1w} in the weak form, and
take their difference, side by side. This leads to the expression
\be\label{tends}\ba{ll}\dy
\int_{\O}\left(\pa_hw^{\ve}_i-\pa_hw_i\right)\pa_h\vp_i\,dx+(2-p)\int_{\Omega}
 \left(c^\ve_{ijhk}-c_{ijhk}\right)\pa^2_{hk}w^{\ve}_j\,\vp_i\,dx\\
 \dy +\,(2-p)\int_{\Omega}
 c_{ijhk}\left(\pa^2_{hk}w^{\ve}_j-\pa^2_{hk}w_j\right)\,\vp_i\,dx,\ea\ee
for any $\vp\in C_0^{\infty}(\Omega)$. The first integral goes to
zero as $\ve$ goes to zero, thanks to the strong convergence of
$w^\ve$ to $w$ in $W^{1,2}(\O)$. Concerning the second integral, we
recall that mollifiers converge in $L^p$ to the mollified function,
as  $\ve$ goes to zero, and that $L^p$ convergence implies almost
everywhere convergence of a subsequence.
Therefore, $c^\ve_{ijhk}$
converges to $c_{ijhk}\,,$ a.e. in $\O$. From \eqref{b}, by
recalling that $\Omega$ is bounded and by using the dominated
convergence theorem, it follows that \be\label{limeps} \lim_{\ve\to
0}\int_{\O}|c^\ve_{ijhk}-c_{ijhk}|^2\,dx=0\,.\ee Hence the second
integral in \eqref{tends} goes to zero. The last integral in
\eqref{tends} tends to zero, thanks to the weak convergence of
$w^\ve$ to $w$ in $W^{2,q}(\O)$, since the coefficients
$c_{ijhk}(x)\,$ are bounded.
\par
Finally, it is easy to verify that $w=\,u\,$. Indeed, by taking the
difference of \eqref{reg1} and \eqref{rega1w}, side by side, and by
setting $V=u-w$, we get
$$\left\{\ba{ll}\vs1
\dy-\Delta V-(p-2) \frac{\nabla u\cdot\nabla\nabla V\cdot \nabla u}
{\left(\mu+\,|\nabla u|\right)\,|\nabla u|}=0, \quad \mbox{ in } \O\,,\\
\dy V=0\,, \quad \mbox{ on }
\partial \O\,.\ea\right . $$
Finally, multiply the above equation by $\Delta V$ and integrate in
$\Omega$. By appealing to arguments already used, one readily
recognizes that, under our assumptions on $p$,  the vector $V$
satisfies $\|\,\Delta V\,\|=0$. Hence $V=0$, by uniqueness.\par
\end{proof}
The Corollary \ref{corollaryq2} is an immediate consequence of
Theorem \ref{teoremaq}. Details are left to the reader.
\section{Appendix}
 Our aim is to show the estimate $$|\,I\,|:=\left|\,\nabla v\cdot
\left( \pa_j\,\nabla v\right) (\pa_j\,v_i)\, \Delta v_i\,\right|\leq
\,|\nabla v|^2\, |D^2\,v|\,|\Delta v|\,.$$ In the sequel, for
convenience, we sometimes avoid the summation convention, by
explicitly writing the sums, even if repeated indexes appear.\par%
We recall that
$$(\,D^2v_k)^2\,:=\sum_{j,h=1}
^3\!\!\left|\,\pa_{jh}^2\,v_k\,\right|^2 \quad\mbox{ and }\quad
|\,D^2v|^2:=\sum_{k=1}^3 (\,D^2v_k)^2\,:=\sum_{k,j,h=1}
^3\!\!\left|\,\pa_{jh}^2\,v_k\,\right|^2.$$
We introduce the vectors
$b$ and $w$, whose components are defined as follows
$$b_j:=(\pa_j\,v)\cdot \, \Delta v\,, \quad w_k^2:=\sum_{j,h=1}^3\left(\,(\pa_h\,v_k)\,b_j\right)^2\,.$$
 The modulus of vector $b$ satisfies the following estimate:
 $$|\,b\,|=\sum_{j=1}^3\, b_j^2\,\leq \sum_{j=1}^3|\,\pa_j\,v|^2|\Delta v|^2=|\Delta v|^2
 \sum_{j=1}^3 \sum_{i=1}^3(\,\pa_j\,v_i)^2=|\Delta v|^2|\nabla v|^2.$$
Hence \be\label{wk} w_k^2=\sum_{h=1}^3\,\left(\pa_h\,v_k\right)^2
\sum_{j=1}^3\,b_j^2=|\nabla v_k|^2|\Delta v|^2|\nabla v|^2\,.\ee
Moreover
$$\ba{ll}\vs1 \dy|\,I\,|=\left|\,\sum_{j,h,k=1}^3\left(\pa_h\,v_k\right)
\left(\pa^2_{hj}\,v_k\right)\,b_j\right|\leq
\sum_{k=1}^3\left|\,\sum_{j,h=1}^3\left(\pa^2_{hj}\,v_k\right)\,(\pa_h\,v_k)\,b_j
\right|\\\hfill \dy \leq \sum_{k=1}^3\sqrt{
\sum_{j,h=1}^3\left(\pa^2_{hj}\,v_k\right)^2}\,\sqrt{\sum_{j,h=1}^3
\left(\,(\pa_h\,v_k)\,b_j\right)^2}\,,\ea$$ where, in the last step,
we have used that, for any pair of tensors $A$ and $B$,  there holds
$|A\cdot B|\leq |A|\,|B|$. Hence, by the above notations and
estimate \eqref{wk}, we get
$$\ba{ll}\vs1\dy|\,I\,|\leq \sum_{k=1}^3|D^2v_k|\,|w_k|\leq
\,|\Delta v|\,|\nabla v|\,\sum_{k=1}^3|D^2v_k|\,|\nabla v_k|
\\
\dy \leq \,|\Delta v|\,|\nabla
v|\sqrt{\sum_{k=1}^3|D^2v_k|^2}\,\sqrt{\sum_{k=1}^3|\nabla v_k|^2}=
\,|\Delta v|\,|\nabla v|^2\,|D^2v|\,,\ea$$ which is our thesis.
  \vskip 0.5cm
\indent
 {\bf Acknowledgments}\,:
The authors like to thank Professor M. Fuchs and Professor P.
Kaplick\'y for giving some interesting references.
  The work of the second author was
supported by INdAM (Istituto Nazionale di Alta Matema\-tica) through
a Post-Doc Research Fellowship.

\end{document}